\documentclass[11pt]{amsart}

\usepackage[a4paper, total={125mm, 190mm}]{geometry}
\usepackage{geometry}
\usepackage[english]{babel}
\usepackage{newlfont}
\usepackage{fancyhdr}
\usepackage{graphicx}
\usepackage{amsmath}
\usepackage{amsfonts}
\usepackage{amsthm}
\usepackage{latexsym}
\usepackage[utf8]{inputenc}
\usepackage{enumerate}
\usepackage{stmaryrd}
\usepackage{tikz}
\usepackage{url}

\theoremstyle{plain}

\newtheorem{thm}{Theorem}[section]
\newtheorem{conj}[thm]{Conjecture}

\newtheorem{lem}[thm]{Lemma}
\newtheorem{prop}[thm]{Proposition}

\theoremstyle{definition}
\newtheorem{defn}{Definition}[section]

\title{On generalized nefness and bigness in adjunction theory}
\author{Camilla Felisetti and Claudio Fontanari}
\date{}

\newcommand{\ita}{\textit}

\thanks{\noindent
\textit{2020 Mathematics Subject Classification.} 14N30, 14J40. \\
\noindent
\textit{Keywords and phrases.} Termination of adjunction, Ambro Conjecture, Serrano Conjecture, strictly nef divisor, $k$-big divisor.}

\begin{document}
\maketitle
\begin{abstract}
We investigate effectiveness and ampleness of adjoint divisors of the form $aL+bK_X$, where $L$ is a suitably positive line bundle on a smooth projective variety $X$ and $a,b$ are positive integers. 
\end{abstract}

\section{Introduction}

Let $X$ be a smooth projective variety of dimension $n$ defined over the complex field $\mathbb{C}$ and let $K_X$ be its canonical divisor. 
Adjunction theory is a classical branch of algebraic geometry investigating positivity properties of adjoint divisors, namely, divisors of the form $aL+bK_X$, where $L$ is a suitably positive line bundle on $X$ (for instance, effective, nef, big, or ample) and $a,b$ are positive integers. In order to give a flavour of the various open questions in the field, we distinguish several cases.   

First, fix $a=b=1$ and consider the effectiveness of 
$L+K_X$. This is the 
content of the following conjecture, first stated by Ambro in \cite{Ambro99}:

\begin{conj}[Ambro]\label{ambro}
If $D$ is an ample divisor on $X$ and $D-K_X$ is nef and big, then $H^0(X,D)\neq 0$.
\end{conj}

In the original statement $D$ was assumed just to be nef, but
(as explained in \cite[p. 176]{Kawamata00}) one can actually reduce to the case where $D$ is ample. Though Conjecture \ref{ambro} is yet to be proven in full generality, some partial answer have been provided by Kawamata \cite{Kawamata00} for surfaces and minimal threefolds, by Broustet-H\"oring \cite{BroustetHoring10} for non-uniruled threefolds, and by H\"oring \cite{Horing12} for all threefolds, even allowing 
$\mathbb{Q}$-factorial canonical singularities.

Now, fix $a=1$ and let $b$ grow and consider the effectiveness 
of $L+bK_X$. This is the classical problem of \emph{Termination of adjunction}, first addressed by Castelnuovo and Enriques (see \cite{CastelnuovoEnriques01}) for algebraic surfaces. More generally, in \cite{AndreattaFontanari18}, the following natural conjecture was formulated: 

\begin{conj}\label{adjterm}
Let $L$ be an effective divisor on $X$.
If {\sl adjunction terminates in the classical sense} for $L$, 
i.e. if there exists an integer $m_0 \ge 1$ such that 
$H^0(X, L+mK_X)=0$ for every integer $m \ge m_0$, then
$X$ is uniruled.
\end{conj}

Indeed, if adjunction terminates in the classical sense for 
some effective divisor $L$, then $X$ has negative Kodaira 
dimension (see \cite[Proposition 2]{AndreattaFontanari18})
and all varieties with negative Kodaira dimension are expected 
to be uniruled (see for instance \cite[(11.5)]{Mori87}).

Finally, fix $b=1$ and let $a$ grow and consider the ampleness 
of $aL+K_X$. This is the content of the following conjecture, stated by Serrano in \cite{Serrano95}:

\begin{conj}[Serrano]\label{serrano}
Let $L$ be a strictly nef line bundle on $X$. Then $K_X+tL$ is ample for $t>n+1$.
\end{conj}

Recall that a divisor $L$ is said to be \ita{strictly nef} if $L\cdot C>0$ for each irreducible curve $C$ on $X$. Conjecture \ref{serrano} has been verified for surfaces 
and for threefolds with some possible exceptions for Calabi-Yau threefolds (see \cite{Serrano95} and \cite{campanachenpeternell08}).  

\smallskip

Here we collect partial results towards these three directions.

First, we introduce a natural notion of $k$-big divisors (see Definition \ref{k-bigness} and Lemma \ref{kbigness}: notice the slight difference from the analogous notion in 
\cite[p. 5]{BeltramettiSommese95}), in such 
a way that a nef divisor $L$ is $1$-big if $L^{n-1} \cdot H > 0$ for some ample 
divisor $H$, and we investigate its basic properties. We may generalize
Conjecture \ref{ambro} in terms of $1$-bigness as follows: 

\begin{conj}[Generalized Ambro conjecture]\label{a+}
Let $D$ be an ample line bundle such that $D-K_X$ is nef and $1$-big. Then $H^0(X,D)\neq 0$.
\end{conj}

We provide the first evidence for this generalized conjecture by verifying it 
for non-uniruled surfaces (see Proposition \ref{ambrononuniruled}) and for minimal threefolds (see Proposition \ref{ambrominimal3}).

Next, we point out that if we add the assumption that $X$ 
is minimal in the sense of the \emph{Minimal Model Program}, 
i.e. we assume that $K_X$ is nef, 
then the problem of termination of adjunction becomes more manageable.
Indeed, in this case 
$h^0(X,aL+bK_X) = \chi(X,aL+bK_X)$ 
for every nef and big divisor $L$\
by Kawamata-Viehweg vanishing theorem.

In dimension up to $5$ we obtain the following result: 

\begin{prop}\label{not}
Let $X$ be minimal of dimension $n \le 5$ and 
let $L$ be an effective ample divisor on $X$.
If $n=5$ assume also that
$(\nu(X), \chi(X, \mathcal{O}_X)) \ne (1,0)$, 
where $\nu(X)$ denotes the numerical dimension of $X$. 
Then adjunction does not terminate in the classical sense for $L$. 
\end{prop}

For minimal varieties of arbitrary dimension, the following weaker statement still generalizes the classical case of surfaces: 

\begin{prop}\label{n-1}
Let $X$ be minimal of dimension $n \ge 2$ and 
let $L$ be an effective nef and big divisor on $X$.
Then adjunction does not terminate in the classical sense for $(n-1)L$.
Furthermore, if n $\ge 3$ and $L$ is base point free and big, then 
adjunction does not terminate in the classical sense also for $(n-2)L$.
\end{prop}

Finally, we turn to Conjecture \ref{serrano} and we 
show that it holds for strictly nef smooth prime divisors on a variety of dimension $n$ if Conjecture \ref{serrano} holds in dimension $n-1$:

\begin{thm}\label{effser}
Assume that Conjecture \ref{serrano} holds in dimension $n-1$ and let $L$ be a smooth prime 
divisor on $X$. If $L$ is strictly nef then $K_X+tL$ is ample for every rational number $t>n+1$.
\end{thm}

Our inductive approach is inspired by the note \cite{Fukuda04}, focusing on 
a special case of Conjecture \ref{serrano}
proposed by Campana and Peternell 
\cite{CampanaPeternell91} a few years before the general formulation by Serrano: 

\begin{conj}[Campana-Peternell] \label{ws}Let $X$ be a smooth projective variety with $-K_X$ strictly nef. Then $-K_X$ is ample, i.e. $X$ is Fano.
\end{conj}

We point out that the main result obtained for 4-dimensional logarithmic pairs $(X,\Delta)$ in \cite{Fukuda04} may be generalized to suitable $n$-dimensional pairs (see Theorem \ref{fukgen}). 

\subsection{Acknowledgements}

We are grateful to Marco Andreatta for many inspiring conversations and to Mauro Varesco for a careful reading of the manuscript.
This research project was partially supported by GNSAGA of INdAM and by PRIN 2017 “Moduli Theory and Birational Classification”.



\section{$k$-big divisors and the generalized Ambro Conjecture}

Let $X$ be a smooth projective variety of dimension $n$. 
\begin{defn}\label{k-bigness}
Fix an integer $0\leq k<n$. We say that a nef divisor $L$ is $k$-big if  
$\nu(L) \ge n-k$, where 
\begin{equation}\label{numericaldim}
\nu(L) = \max \{d: L^dH^{n-d} > 0 \textrm{ for some ample divisor } H \}
\end{equation}
denotes the numerical dimension of a nef divisor $L$.
\end{defn}

Notice that by definition any $0$-big divisor is big, exactly as any $0$-ample divisor 
in the sense of \cite{Totaro13} is ample. More generally, the following holds:

\begin{lem}\label{kbigness}
A nef divisor $L$ is $k$-big for some integer $k$ with $0\leq k<n$ if and only if 
$$
L^{n-k}\cdot H_1\cdot H_2\cdot\ldots\cdot H_k>0
$$
for some ample divisors $H_1,\ldots, H_k$. 
\end{lem}

\begin{proof}
To check the "if" part, assume by contradiction that $L^{n-k}H^{k}=0$ for any ample divisor $H$. 
By \cite[II, Proposition 6.3]{Nakayama04}, we have that $L^{n-k}$ is 
homologically equivalent to zero, hence 
$L^{n-k}\cdot H_1\cdot H_2\cdot\ldots\cdot H_k=0$
for any choice of ample divisors $H_1,\ldots, H_k$.
Conversely, the "only if" part is straightforward (just take 
$H_1 = \ldots = H_k = H$).
\end{proof}

In \cite{BeltramettiSommese95}, p. 5, a slightly different notion 
of $k$-bigness is introduced. By \cite[Lemma 2.5.8]{BeltramettiSommese95}, 
if a nef divisor is $k$-big in their sense then it is 
$k$-big in ours too, in particular a $k$-ample divisor 
is $k$-big in both senses (see \cite{BeltramettiSommese95}, p. 44). 
On the other hand, \cite[Example 1.8.2.2]{BeltramettiSommese95} shows that the two definitions of $k$-bigness do not coincide.

Our notion of $k$-bigness satisfies the following monotonicity property.

\begin{prop}
If L is $k$-big then it is $(k+1)$-big.
\end{prop}
\begin{proof}
If $L$ is a $k$-big divisor, then by Lemma \ref{kbigness} there exist ample divisors $H_1,\ldots, H_k$ such that $L^{n-k}\cdot H_1\cdot\ldots\cdot H_k>0$.
Moreover since $L$ is nef, it can be written as a limit 
$$L=\lim_{t\to +\infty} H_t, \quad \text{ with }H_t \text{ ample.}$$
Then we have 
$$0<L^{n-k-1}\cdot L\cdot H_1\cdot\ldots\cdot H_k=L^{n-k-1}\cdot\lim_{t\to +\infty} H_t\cdot H_1\cdot\ldots\cdot H_k,$$
which implies that for $t\gg 0$
$$L^{n-k-1}\cdot H_t\cdot H_1\cdot\ldots\cdot H_k>0,$$
so by Lemma \ref{kbigness} $L$ is $(k+1)$-big. 
\end{proof}

Definition \ref{k-bigness} is motivated by the following natural generalization of Kawamata-Viehweg vanishing theorem.
Despite the proof is well-known (see for instance \cite[Lemma 2.1]{LazicPeternell17}), we report it here for completeness sake.

\begin{prop}\label{kvanishing}
Let $L$ be $k$-big divisor. Then 
$$ H^i(X,K_X+L)=0\quad \forall i>k.$$
\end{prop}
\begin{proof}
If $k=0$ then $L$ is nef and big, so the result follows from Kawamata-Viehweg vanishing theorem. Assume now $k \ge 1$. 
We argue by induction on the dimension of $X$.
If $\dim X=1$ then the claim is true.
Suppose now that the result is true for any $k$-big divisor with $k\geq 1$ on varieties of dimension $\leq n-1$. 
If $L$ is a $k$-big divisor, then there exist ample divisors $H_1,\ldots H_k$ such that $L^{n-k}\cdot H_1\cdot\ldots\cdot H_k>0$.
Up to multiples, we may assume that $H_k$ is very ample and smooth, so we can restrict to 
$H_k$ and we have
$$0<L^{n-k}\cdot H_1\cdot\ldots\cdot H_k=(L^{(n-1)-(k-1)}H_1\cdot \ldots \cdot H_{k-1})_{\mid H_k}.$$
This implies that $L$ is $(k-1)$-big when restricted to $H_k$.
\smallskip
Consider now the short exact sequence 
$$0\rightarrow \mathcal{O}_X(K_X+L)\rightarrow \mathcal{O}_X(K_X+L+H_k)\rightarrow \mathcal{O}_{H_k}(K_H+L_{\mid H_k})\rightarrow 0,$$
which induces the long exact sequence in cohomology
\begin{equation}
    \ldots \rightarrow H^i(X,K_X+L)\rightarrow H^i(X,K_X+L+H_k)\rightarrow H^i(H_k,K_H+L_{\mid H_k})\rightarrow \ldots
\end{equation}
Since $L+H_k$ is the sum of an ample and a nef divisor, hence it is ample, by Kodaira vanishing theorem we have $$H^i(X,K_X+L+H_k)=0\quad \forall i>0,$$
which yields $H^{i+1}(X,K_X+L)\cong H^i(H_k,K_{H_k}+L_{\mid H_k})$ for all $i>0$.
Moreover, by induction we have 
$$H^i(H_k,K_{H_k}+L_{\mid H_k})=0\quad \forall i>k,$$
which implies the result.
\end{proof}

Finally we provide the first evidences to Conjecture \ref{a+}. 

\begin{prop}\label{ambrononuniruled}
Let $X$ be a non-uniruled surface. Then Conjecture \ref{a+} holds.
\end{prop}
\begin{proof}
If $D-K_X$ is $1$-big, then by Proposition \ref{kvanishing} we have $H^i(X,D)=0$ for $i=2$. 
As a result, $\chi :=\chi(X,D)=H^0(X,D)-H^1(X,D)$, which implies $H^0(X,D)\geq \chi$.
By Riemann-Roch we have 
$$\chi=\dfrac{1}{2}D\cdot(D-K_X)+\chi(X,\mathcal{O}_X).$$
Since $D-K_X$ is 1-big and $D$ is ample, the first term in the sum is positive, thus $\chi>\chi(X,\mathcal{O}_X)$. The result follows as for any non-uniruled surface $X$ we have $\chi(X,\mathcal{O}_X)\geq 0$ (see for instance \cite[Theorem X.4]{Beauville10}).
\end{proof}

\begin{prop}\label{ambrominimal3}
Let $X$ be a minimal threefold. Then Conjecture \ref{a+} holds.
\end{prop}

\begin{proof}
If $D-K_X$ is $1$-big, then by Proposition \ref{kvanishing}
we have $H^i(X,D)=0$ for all $i\geq2$. 
As before, we have $\chi :=\chi(X,D)=H^0(X,D)-H^1(X,D)$, which implies $H^0(X,D)\geq \chi$.
Arguing as in the proof of \cite[Proposition 4.1]{Kawamata00} we have 
\begin{eqnarray}\label{chil}
 &\qquad \chi&=\dfrac{1}{12}(2D-K_X)\cdot\left(\dfrac{1}{6}D^2+\dfrac{2}{3}D\cdot(D-K_X)+\dfrac{1}{6}(D-K_X)^2 \right)+\\
 & &\ + \dfrac{1}{72}(2D-K_X)\cdot(3c_2(X)-K_X^2)+\dfrac{1}{24}K_X\cdot c_2(X)+\chi(X,\mathcal{O}_X)
\end{eqnarray}
with $\chi(X,\mathcal{O}_X)+\dfrac{1}{24}K_X\cdot c_2(X)\geq 0$ and $3c_2(X)-K_X^2$ pseudoeffective by \cite{Miyaoka87}.
Observe that $2D-K_X=D+D-K_X$ is nef since it is the sum of two nef divisors. Then
$$\dfrac{1}{72}(2D-K_X)\cdot(3c_2(X)-K_X^2)\geq 0,$$
since it is the intersection of a nef divisor and a pseudoeffective curve. 
Now we write the first term of \eqref{chil} as 
\begin{align*}
    &\dfrac{1}{72}D^3+ \dfrac{1}{12}D\cdot \left(\dfrac{2}{3}D\cdot(D-K_X)+\dfrac{1}{6}(D-K_X)^2 \right)+\\
    &+\dfrac{1}{12}(D-K_X)\left(\dfrac{1}{6}D^2+\dfrac{2}{3}D\cdot(D-K_X)+\dfrac{1}{6}(D-K_X)^2 \right).
\end{align*}
Since $D$ is ample we have $D^3>0$, while the rest of the sum is non negative as it is given by intersection products of nef divisors. 
As a consequence, we have $h^0(X,D)\geq \chi>0$. 
\end{proof}

\section{Termination of adjunction for minimal varieties}

A Hirzebruch-Riemann-Roch computation yields 
the following auxiliary result. As pointed out by one of the referees, our argument is very similar to the proof of \cite[Theorem 4.1]{LazicPeternell17} by Lazi\'c and Peternell, which works also for varieties with terminal singularities. 

\begin{lem}\label{main} Let $X$ be a minimal smooth projective variety of dimension $n \ge 3$. If adjunction terminates in the classical sense for an effective nef and big divisor $L$ on $X$ then we have $K_X^i L^{n-i}=0$ for $n-3 \le i \le n$.  
\end{lem}

\proof Since by assumption $K_X$ is nef then $(m-1)K_X+L$ is nef and big for every $m \ge 1$, hence by Kawamata-Viehweg vanishing theorem 
we have $\dim H^0(X, mK_X+L)=\chi(X, mK_X+L)$, 
a polynomial in the variable $m$ of degree at most $n$. Therefore, if adjunction terminates 
in the classical sense for $L$, then $\chi(X, mK_X+L)$ must be identically zero. 
On other hand from Hirzebruch-Riemann-Roch it follows that 
\begin{eqnarray*}
\chi(X, mK_X+L) &=& \frac{(mK_X+L)^n}{n!} - \frac{1}{2} K_X \frac{(mK_X+L)^{n-1}}{(n-1)!} \\
& & + \frac{1}{12}(K_X^2+c_2) \frac{(mK_X+L)^{n-2}}{(n-2)!} \\
& & - \frac{1}{24} K_X c_2 \frac{(mK_X+L)^{n-3}}{(n-3)!} + o(m^{n-3})\\
&=& A m^n + B m^{n-1} + C m^{n-2} + D m^{n-3} + \ldots
\end{eqnarray*}
with 
\begin{eqnarray*}
A &=& \frac{K_X^n}{n!}\ , \\
B &=& n \frac{K_X^{n-1}L}{n!} - \frac{1}{2} K_X \frac{K_X^{n-1}}{(n-1)!}\ , \\
C &=& \frac{n(n-1)}{2} \frac{K_X^{n-2}L^2}{n!} - \frac{1}{2} (n-1) K_X \frac{K_X^{n-2}L}{(n-1)!}\ , \\
& &+ \frac{1}{12}(K_X^2+c_2) \frac{K_X^{n-2}}{(n-2)!} \\
D &=& \frac{n(n-1)(n-2)}{6}  \frac{K_X^{n-3}L^3}{n!}  - \frac{1}{2} \frac{(n-1)(n-2)}{2} K_X \frac{K_X^{n-3}L^2}{(n-1)!} \\
& &+ \frac{1}{12}(n-2)(K_X^2+c_2) \frac{K_X^{n-3}L}{(n-2)!} - \frac{1}{24} K_X c_2 \frac{K_X^{n-3}}{(n-3)!}\ .
\end{eqnarray*}
Now from $A=B=0$ we get immediately $K_X^n = K_X^{n-1}L = 0$. Next we set $C=0$: since $K_X$ is nef $K_X^{n-2}L^2 \ge 0$ 
and by Miyaoka's inequality (see for instance \cite[Theorem 3.12]{MiyaokaPeternell97}) $c_2 K_X^{n-2}\ge 0$, hence we deduce 
$K_X^{n-2}L^2 = 0$ and $c_2 K_X^{n-2}= 0$. Finally, let $D=0$: once again, $K_X^{n-3}L^3 \ge 0$ since $K_X$ is nef and 
$c_2 K_X^{n-3}L \ge 0$ by Miyaoka's inequality, hence $K_X^{n-3}L^3 = 0$ as claimed. 
\qed

\smallskip

Recall that, when $K_X$ is nef, one sets $\nu(X) := \nu(K_X) \ge 0$, where the numerical dimension $\nu$ of a nef divisor is defined as in \eqref{numericaldim}.
\smallskip


\emph{Proof of Proposition \ref{not}.} 
If $n=3$ the claim is immediate from Lemma \ref{main} since 
$L^n > 0$ for every nef and big divisor $L$. 
Next, observe that Lemma \ref{main} and Lemma \ref{kbigness} imply $\nu(X)=0$ for $n=4$ and $\nu(X)\leq 1$ for $n=5$.
Now, if $n=4$ or $n=5$ and $\nu(X)=0$, then by \cite{Kawamata13}
the Kodaira dimension of $X$ would be zero, contradicting \cite[Proposition 2]{AndreattaFontanari18}.
Finally, if $n=5$ and $\nu(X)=1$, then by assumption $\chi(X, \mathcal{O}_X) \ne 0$ and by \cite[Theorem B, (ii)]{LazicPeternell18}, the Kodaira dimension of $X$ would be nonnegative, contradicting 
\cite[Proposition 2]{AndreattaFontanari18}. 
\qed

\smallskip

In the same vein, from the results of \cite{Horing12} we may easily deduce Proposition \ref{n-1}.

\smallskip

\emph{Proof of Proposition \ref{n-1}.} Since by assumption $K_X$ is nef then $(m-1)K_X+aL$ is nef and big for every $m \ge 1$ and 
$a \ge 1$, hence by Kawamata-Viehweg vanishing theorem 
we have $\dim H^0(X, mK_X+aL)=\chi(X, mK_X+aL)$, a polynomial in the variable $m$ of degree at most $n$. Therefore, 
if we assume by contradiction that adjunction terminates in the classical sense for $aL$, then $\chi(X, mK_X+aL)$ must be identically zero. 

In order to prove the first statement we just recall that
$\chi(X, K_X+(n-1)L) = \dim H^0(X, K_X+(n-1)L) \ne 0$ by 
\cite[Theorem 1.2]{Horing12}, and this 
contradiction ends the proof.

In order to prove the second statement we check by induction on $n$
that $H^0(X, \mathcal{O}_X(K_X+(n-2)L)) \ne 0$, hence getting a 
contradiction. The case $n=3$ holds by \cite[Theorem 1.5]{Horing12}. Assume now $n \ge 4$. 
By Bertini theorems the general element $L$ of 
$\vert L \vert$ is a smooth projective variety of dimension $n-1$ and by adjunction 
we have $ (K_X + (n-2)L)_{\vert L}=K_L + (n-3) L_{\vert L}$. The exact sequence: 
$$
0 \to \mathcal{O}_X(K_X+(n-3)L) \to \mathcal{O}_X(K_X+(n-2)L) \to \mathcal{O}_H(K_L+(n-3)L_{\vert L}) \to 0
$$
induces the exact sequence in cohomology: 
\begin{eqnarray*}
\ldots &\to& H^0(X, \mathcal{O}_X(K_X+(n-2)L)) \to H^0(L, \mathcal{O}_L(K_L+(n-3)L_{\vert L}) \to \\
&\to& H^1(X, \mathcal{O}_X(K_X+(n-3)L)) \to \ldots
\end{eqnarray*}
By the Kodaira vanishing theorem we have $H^1(X, \mathcal{O}_X(K_X+(n-3)L)) = 0$ and by induction we have 
$H^0(X, \mathcal{O}_L(K_L+(n-3)L)) \ne 0$, hence we deduce $H^0(X, \mathcal{O}_X(K_X+(n-2)L)) \ne 0$ as well.  
\qed


\section{Towards Serrano conjecture}

First we prove that if Conjecture \ref{serrano} 
holds in dimension $n-1$ then it holds in 
dimension $n$ for smooth prime divisors.

\smallskip

\emph{Proof of Theorem \ref{effser}.}
By adjunction we have
$$(K_X+tL)_{\mid L}=\left(K_X+L+(t-1)L\right)_{\mid L}=K_L+(t-1)L_{\mid L},$$
which is ample since Conjecture \ref{serrano} is 
assumed to hold in dimension $n-1$. 
This implies that $(K_X+tL)_{\mid L}$ is big, i.e.
$$\left((K_X+tL)_{\mid L}\right)^{n-1}=(K_X+tL)^{n-1}\cdot L>0.$$
Hence we deduce 
$$0 < (K_X+tL)^{n-1}\cdot (K_X+tL-K_X)=(K_X+tL)^{n}+ (-K_X) \cdot (K_X+tL)^{n-1}.$$
In order for the above sum to be strictly positive, we must have either $(K_X+tL)^{n}>0$ or 
$(-K_X) \cdot (K_X+tL)^{n-1} > 0$. 
Notice that $(K_X+tL)^{n} \ge 0$ by \cite[Lemma 1.1]{Serrano95}.
If $(K_X+tL)^{n}>0$ then $K_X+tL$ is ample by \cite[Lemma 1.3]{Serrano95}.
If instead $(-K_X) \cdot (K_X+tL)^{n-1}>0$ and $(K_X+tL)^{n}=0$ then $K_X+tL$ is not strictly nef by \cite{Matsuki94} (see also \cite[Lemma 1.2]{Fukuda04}), 
contradicting \cite[Lemma 1.1]{Serrano95}.
\qed

\smallskip

Let now $\Delta$ be a reduced Cartier divisor on $X$, with $\Delta=\sum_{i=1}^k\Delta_i$.

\begin{defn}\label{simple}
A Cartier divisor $\Delta=\sum_{i=1}^k\Delta_i$ on $X$ is called simple normal crossing if it is reduced and for all $J\subset \{1, \ldots, k \}$ the subvariety $\Delta_J=\bigcap_{j\in J} \Delta_j$ is smooth with irreducible components of codimension $\vert J\vert$.
\end{defn}

In what follows, by adapting Fukuda's argument in \cite{Fukuda04} for fourfolds with normal 
crossing boundary, we provide a log-version of Conjecture \ref{ws} for varieties of 
arbitrary dimension with simple normal crossing boundary.

\begin{thm}\label{fukgen}
Let $X$ be a smooth variety of dimension $n$ and let $\Delta=\sum_{j=1}^k\Delta_j$ be a simple normal crossing divisor on $X$ with $k>n-4$. Then 
\begin{enumerate}[(i)]
\item $K_X+\Delta$ strictly nef $\Rightarrow$ $K_X+\Delta$ ample;
\item $-(K_X+\Delta)$ strictly nef $\Rightarrow$ $-(K_X+\Delta)$ ample.
\end{enumerate}
\end{thm}

As in \cite{Fukuda04}, given a divisor $\Delta$ one defines $Strata(\Delta)$ as the set
$$Strata(\Delta):=\left\lbrace \Gamma\mid \Gamma \text{ is an irreducible component of }\Delta_J=\bigcap_{j\in J} \Delta_j\neq \emptyset \right\rbrace,$$
and $MS(\Delta)$ as the set of $\Gamma\in Strata(\Delta)$ which are minimal in the following sense: if $\Gamma'\subset \Gamma$ is a stratum of $\Delta$ then $\Gamma'=\Gamma$.
As a direct consequence of Definition \ref{simple}, if $\Delta$ is a simple normal crossing divisor then $MS(\Delta)$ coincides with the irreducible components of $\bigcap_{j=1}^k \Delta_j$ of maximal codimension (see for instance \url{https://stacks.math.columbia.edu/tag/0CBN}).

\smallskip

\emph{Proof of Theorem \ref{fukgen}.}
To prove assertion $(i)$, notice that since $K_X+\Delta$ is strictly nef so is $(K_X+\Delta)_{\mid \Gamma}=K_{\Gamma}$ for all $\Gamma\in MS(\Delta)$.
The assumption $k>n-4$ implies $\dim\Gamma<4$, so Conjecture \ref{serrano} holds and $(K_X+\Delta)_{\mid \Gamma}$ is ample for all $\Gamma\in MS(\Delta)$.
Now the result follows by \cite[Theorem 1.5]{Fukuda04} with $L :=K_X+\Delta$ and 
\cite[Lemma 1.4]{campanachenpeternell08}
The proof of assertion $(ii)$ is completely analogous. 
\qed

\bibliographystyle{siam}
\bibliography{main}

\vspace{1cm}

\noindent
Camilla Felisetti, Claudio Fontanari \\
Dipartimento di Matematica \\
Universit\`a degli Studi di Trento \\
Via Sommarive 14, 38123 Trento (Italy) \\
camilla.felisetti@unitn.it, claudio.fontanari@unitn.it

\end{document}